\documentclass[letterpaper,10pt,pdflatex,reqno]{amsart}
\usepackage{amsmath,times} \usepackage{amsthm}
\usepackage[textsize=tiny]{todonotes}
\usepackage{graphicx}
\usepackage{color}

\usepackage{amssymb} \usepackage{latexsym} \usepackage{amscd}
\usepackage{verbatim} 

\theoremstyle{definition} 
\newtheorem{theorem}{Theorem}[section]
\newtheorem{theorem*}{Theorem}
\newtheorem{lemma}[theorem]{Lemma}
\newtheorem{proposition}[theorem]{Proposition}
\newtheorem{corollary}[theorem]{Corollary}
\newtheorem{innercustomthm}{Theorem}
\newenvironment{customthm}[1]
  {\renewcommand\theinnercustomthm{#1}\innercustomthm}
  {\endinnercustomthm}

\newtheorem{definition}[theorem]{Definition}
\newtheorem*{definition*}{Definition}
\newtheorem{example}[theorem]{Example}
\newtheorem{remark}[theorem]{Remark}
\newtheorem*{remark*}{Remark}

\newcommand{\Diff}{\operatorname{Diff}}

\begin{document}

\begin{abstract}
We give two applications of the the duality between the Homogeneous Complex Monge-Amp\`ere Equation (HCMA) and the Hele-Shaw flow.  First, we prove existence of smooth boundary data for which the weak solution to the Dirichlet problem for the HCMA over $\mathbb P^1\times \overline{\mathbb D}$ is not twice differentiable at a given collection of points, and also examples that are not twice differentiable along a set of codimension one in $\mathbb{P}^1\times \partial \mathbb{D}$.   Second, we discuss how to obtain explicit families of smooth geodesic rays in the space of K\"ahler metrics on on $\mathbb P^1$ and on the unit disc $\mathbb D$ that are constructed from an exhausting family of increasing smoothly varying simply connected domains.  
\end{abstract}

\title[Applications of Duality of Hele-Shaw flow]{Applications of the duality between the Complex Monge-Amp\`ere Equation and the Hele-Shaw flow}

\author{Julius  Ross and David Witt Nystr\"om}
\maketitle
\section{Introduction}
The purpose of this paper is to give two applications of previous work of the authors that describes a duality between a certain Dirichlet problem for the Homogeneous Complex Monge-Amp\`ere Equation (HCMA) and a free boundary problem in the plane called the Hele-Shaw flow \cite{RW}.    First, for any finite set of points in $\mathbb{P}^1\times \partial \mathbb{D}$, where $\mathbb D\subset \mathbb C$ denotes the unit disc, we give examples of smooth boundary data for which the weak solution to this Dirichlet problem over $\mathbb P^1\times \overline{\mathbb D}$ is not twice differentiable at these points. We also produce such examples that are not twice differentiable along a set of codimension one in $\mathbb{P}^1\times \partial \mathbb{D}$.   Second, we use this duality to produce families of regular solutions to this Dirichlet problem over the punctured disc $\overline{\mathbb D}^{\times}$, giving explicit families of smooth geodesic rays in the space of K\"ahler metrics on $\mathbb P^1$ and on $\mathbb D$.   

\subsection{Regularity of the Dirichlet problem for the HCMA over the disc}

The setup for the first application is as follows.    Fix a chart $0\in \mathbb C\subset \mathbb P^1$ with coordinate $z$ and let $\omega$ denote the Fubini-Study form.   Choose  a K\"ahler potential $\phi\in C^{\infty}(\mathbb P^1)$, by which we mean $\omega+dd^c\phi$ is a strictly positive form (i.e.\ an area form) and let $\pi_{\mathbb P^1}:\mathbb P^1\times\overline{\mathbb D}\to \mathbb P^1$ be the projection.   Consider the envelope
\begin{equation}
 \Phi := \sup 
 \left\{
\begin{array}{c}
  \psi \colon \mathbb P^1\times \overline{\mathbb D}\to \mathbb R\cup \{-\infty\} : \psi \text{ is usc, }\pi_{\mathbb P^1}^*\omega + dd^c\psi \ge 0 \\\text{ and } \psi(z,\tau)\le \phi(\tau z) \text{ for } (z,\tau)\in \mathbb P^1\times  \partial\mathbb D
 \end{array}
 \right\},\label{eq:weakD}
\end{equation}
which is the weak solution to the Dirichlet problem
\begin{align*}
\Phi(z,\tau) = \phi(\tau z) &\text{ for }(z,\tau)\in \mathbb P^1\times \partial \mathbb D,\\
\pi_{\mathbb P^1}^*\omega + dd^c\Phi&\ge 0, \\
(\pi_{\mathbb P^1}^* \omega + dd^c\Phi)^{2}&=0,
\end{align*}
where the second equation is to be understood in the sense of currents and the third is the Bedford-Taylor product.

The following is a preliminary version of what we shall prove:

\begin{customthm}{A}\label{thm:notc2:preliminary}
Let $S$ be a union of finitely many points and non-intersecting smooth curve segments in $\mathbb P^1\setminus \{0\}$.  Then there exists a K\"ahler potential such that the above weak solution $\Phi$  to the HCMA is not twice differentiable at any point of the form $(\tau^{-1} z,\tau), z\in S, |\tau|=1$. 
\end{customthm}

The question of regularity of solutions to the HCMA has a long history, and has proved to be a difficult problem that depends subtly on the boundary data (see, for example, Lempert \cite{Lempert}, Bedford-Demailly \cite{Bedford} or B\l ocki \cite{Blocki}).  As is well known, if $\overline{\mathbb D}$ is replaced by a closed annulus in $\mathbb C$, and $\mathbb P^1$ is replaced by any K\"ahler manifold $X$,  the above Dirichlet problem with $S^1$-invariant boundary data corresponds to finding a geodesic segment in the space of K\"ahler potentials on $X$ (with respect to the Mabuchi metric).  Similarly if $\overline{\mathbb D}$ is replaced by the punctured disc $\overline{\mathbb D}^{\times}$ it corresponds to finding a geodesic ray.  The regularity of these geodesics has been of intense interest ever since this space was considered by Mabuchi \cite{Mabuchi}, Semmes \cite{Semmes} and Donaldson \cite{Donaldson}.   However it is only since the relatively recent work of Lempert-Vivas \cite{LempertVivas}, Lempert-Darvas \cite{LempertDarvas} and Darvas \cite{Darvas} that we have known that it is not always possible to join two potentials by a geodesic segment that lies in the class $C^2$.

What we have here is similar in spirit to, but in a sense stronger than, the result of Lempert-Vivas in that we are able to prescribe the location of the singular locus (which need not consist of isolated points), as well as see exactly how the regularity fails;   we are not aware of any similar result in the theory of the HCMA in which this precise information about the weak solution is available, other than the toric case \cite{RubinsteinZelditchI,RubinsteinZelditchII}.    We remark also that in the work of \cite{RubinsteinZelditchIII} irregularity of some geodesics is proved, albeit for a rather different initial value problem.  \medskip 

What permits us to have such a good understanding of the singularities of $\Phi$ is the connection with the Hele-Shaw flow.    To define this, suppose $(X,\omega)$ is a one-dimensional K\"ahler manifold, which we will take to be either $\mathbb P^1$ with its Fubini-Study form $\omega_{FS}$, $\mathbb{C}$ with the Lebesgue form $dA$ or the open unit disc $\mathbb D\subset \mathbb C$ with the Poincar\'e form $\omega_P$.  In the first case we use the convention that $\mathbb P^1$ has area one so
$$ A:= \int_X \omega \in \{1,\infty\}.$$
 The complex plane and unit disc both have the origin as a distinguished point, and when $X=\mathbb P^1$ we fix a point that we denote by $0\in \mathbb P^1$.

Given any $\phi\in C^{\infty}(X)$ such that $\omega + dd^c\phi>\epsilon \omega$ for some $\epsilon>0$, the Hele-Shaw flow consists of an increasing collection of sets
$$ \Omega_t \subset X \text{ for } t\in (0,A)$$
such that $\Omega_t$ has area $t$ with respect to $\omega + dd^c\phi$.   It is defined by setting
$$ \Omega_t : = \{ z\in X : \psi_t(z)<\phi(z)\}$$ where
$$\psi_{t} := \sup\{\psi: \psi \text{ is $\omega$-psh and }\psi\le \phi \text{ and } \nu_{0}(\psi)\ge t\}.$$
Recall that a function $\psi:X\to \mathbb R\cup \{-\infty\}$ is $\omega$-psh if $\psi+u$ is plurisubharmonic whenever locally $dd^c u=\omega$, thus $\omega+dd^c\psi$ will be a positive current, and $$\nu_0(\psi)= \sup\{t : \psi\le t\ln |z|^2 + O(1) \text{ near } z=0\}$$ is the order of the logarithmic singularity (Lelong number) of $\psi$ at $0\in X$.  \medskip

What is proved in \cite{RW} is that when $X=\mathbb{P}^1$ this flow is intimately connected to the weak solution $\tilde{\Phi}$ to the Dirichlet problem for the complex HCMA on $\mathbb{P}^1\times \overline{\mathbb D}^{\times}$ with boundary data the pullback of $\phi$ to $X\times \partial \mathbb D$ and a certain prescribed singularity at $(0,0)$;  in fact $\psi_t$ is the Legendre transform of $\tilde{\Phi}$.  Moreover there is a simple way to transform between $\Phi$ and $\tilde{\Phi}$, and thus each contain the same information as the Hele-Shaw flow (we shall recall this in more detail in Section \ref{sec:HCMA}).\medskip

To state our first result more precisely we need to consider flows of sets that develop singularities in a particularly simple way.   Let $S$ be the union of finitely many points and non-intersecting smooth curve segments in $\mathbb{P}^1\setminus \{0\}$.

\begin{definition}\label{def:tangency}
 We say that the Hele-Shaw for $\omega + dd^c\phi$ \emph{develops tangency} along $S$ if there exists a $T\in(0,1)$ such that (1) $\Omega_t$ is smoothly bounded, simply connected and varies smoothly for $t<T$ and (2) $\Omega_T$ is simply connected and $\partial \Omega_T$ is the image of a smooth locally embedded curve intersecting itself tangentially precisely along $S$ (see Figure \ref{fig1}).
\end{definition}

\begin{figure}[htb]
	\centering
\scalebox{.7}{\input{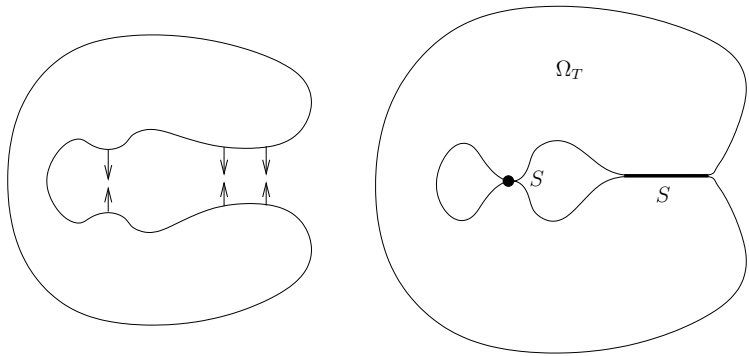}}
	\caption{Developing tangency along $S$}
	\label{fig1}
\end{figure}

\begin{customthm}{B}\label{thm:notc2}
Let $\phi\in C^{\infty}(\mathbb P^1)$ be a K\"ahler potential and suppose the Hele-Shaw flow for $\omega + dd^c\phi$ develops tangency along $S$.   Then the weak solution $\Phi$ from \eqref{eq:weakD} to the Dirichlet problem for the HCMA on $\mathbb P^1\times \overline{\mathbb D}$ with boundary data $(z,\tau)\mapsto \phi(\tau z)$ is not twice differentiable at the points $(\tau^{-1}z,\tau)$, $z\in S, |\tau|=1$.
\end{customthm}

We note that actually we know more, and from the discussion below it will be apparent that there is an explicit open set in $\mathbb P^1\times \overline{\mathbb D}$ on which $\Phi$ smooth.   With more work it may be possible to describe precisely where (and how) $\Phi$ fails to be twice differentiable,  but we shall not consider that further in this paper.

It remains to comment that to get Theorem \ref{thm:notc2:preliminary} from Theorem \ref{thm:notc2} we have to show that given such a set $S$ it is possible to find a K\"ahler potential whose Hele-Shaw flow develops tangency along $S$.  To do this we first choose  $\Omega_T$  as in Definition \ref{def:tangency}, and we aim to find a K\"ahler potential for which we can understand the Hele-Shaw flow backwards for a small time, say for $t\in [T-\epsilon,T]$.  As is well known in the Hele-Shaw literature it is not normally the case that the strong Hele-Shaw flow exists backwards in time starting at some $\Omega_T$ (for instance if $\omega$ is analytic then a necessary condition is that $\Omega_T$ has analytic boundary).  However, using a previous result of the authors \cite{RWDisc}, we shall see that this assumption is not necessary as long as one allows a (smooth) modification of the area form near $\Omega_T$ (said another way, we make a smooth modification of the permeability that governs the flow).   We may then shrink $\Omega_{T-\epsilon}$ down to 0, and expand $\Omega_T$ out to $\infty$,  so as to obtain a flow of sets $\{\Omega_t\}_{t\in (0,1)}$ with properties that ensure that it is the Hele-Shaw flow for some K\"ahler potential that can be constructed from the flow.  Details can be found in Section \ref{sec:designer}.

\subsection*{Families of Geodesic Rays} 

The weak solution $\tilde{\Phi}$ to the Dirichlet problem for the HCMA over the punctured disc $\overline{\mathbb D}^{\times}$ with $S^1$-invariant boundary data  is by definition a weak geodesic ray in the space of positive potentials on $X$.    When $X=\mathbb{P}^1$, if this solution is regular (by which we mean it is smooth and strictly $\omega$-plurisubharmonic along the fibres over $\overline{\mathbb D}^{\times}$) then it gives a genuine geodesic in this space, i.e.\ a smooth geodesic in the space of K\"ahler metrics.     For this reason regularity of the weak geodesic ray is of interest, and following \cite{RW} we know that this regularity is intimately related to the topology of the Hele-Shaw flow. By analogy, when $X=\mathbb{D}$ we call a regular weak geodesic ray a smooth geodesic ray.

To state our theorems in the simplest way, let let $B(t)$ denote the geodesic ball in $\mathbb{P}^1$ centred at $0$ with area $t$ taken with respect to the metric $\omega_{FS}$.

\begin{definition*}
Let $a\in \{0,1\}$.  We say that a collection of subsets $\{\Omega_t\}_{t\in (0,1)}$ of $\mathbb{P}^1$ is \emph{standard as $t$ tends to $a$} if there exist $\epsilon>0$ such that
$$ \Omega_t = B(t) \text{ for } |t-a|<\epsilon.$$
\end{definition*}

\begin{customthm}{C}\label{thm:main}
Let $X=\mathbb P^1$ or $ X=\mathbb D$ and suppose the Hele-Shaw flow $\{\Omega_t\}_{t\in (0,A)}$  for a K\"ahler form $\omega + dd^c\phi$ satisfies
\begin{enumerate}
\item $\{\Omega_t\}_{t\in (0,A)}$ is smoothly bounded and varies smoothly with non-vanishing normal velocity,
\item $\Omega_t$ is simply connected for all $t\in (0,A)$,
\item if $X=\mathbb P^1$ then $\{\Omega_t\}_{t\in (0,1)}$ is standard as $t$ tends to $1$.
\end{enumerate}
Then the weak geodesic ray obtained as the Legendre transform of the Hele-Shaw envelopes  $\{\psi_t\}$  is regular, and so defines a smooth geodesic ray in the space of K\"ahler metrics on $X$.
\end{customthm}

Of course, for this theorem to have any content we must be able to provide examples of potentials $\phi$ for which the Hele-Shaw has these properties.  An interesting case of this is given by a result of Hedenmalm-Shimorin (see also \cite{HedenmalmOlofsson} for the same statement with weaker curvature assumptions).

\begin{customthm}{HS}(Hedenmalm-Shimorin \cite{Hedenmalm})\label{thm:HS}
Let $(X,\omega)=(\mathbb D,\omega_P)$ and suppose that $\phi$ is taken so that the K\"ahler form $\omega_P + dd^c\phi$ is analytic, hyperbolic (i.e. the K\"ahler metric has negative curvature) and complete (e.g. if $\omega_P + dd^c\phi>\epsilon\omega_P$).   Then the Hele-Shaw flow $\{\Omega_t\}$ for $\omega+dd^c\phi$ is smoothly bounded, smoothly varying, and simply connected for all $t\in(0,\infty)$.
\end{customthm}

Another class of examples can be constructed from an observation due to Berndtsson (following a question of Zelditch) which says that \emph{any} reasonable smooth increasing family of simply connected domains is the Hele-Shaw flow for some smooth K\"ahler potential, see Theorem \ref{thm:existencepotential}.

Of course it is trivial to construct families of domains $\{\Omega_t\}$ that satisfy the hypotheses of Theorem \ref{thm:main},  and thus we have an easy way to construct explicit families of smooth geodesic rays in the space of K\"ahler metrics on $\mathbb P^1$ (resp.\ on $\mathbb D$).  In particular we have that any hyperbolic analytic K\"ahler metric with $\omega_P+dd^c\phi>\epsilon\omega_P$ on $\mathbb D$ is the starting point for some canonical smooth geodesic ray.




\subsection*{Acknowledgements } 
We wish to thank the Simons Center for Geometry and Physics for inviting the authors to the ``Large $N$ Program" in the Spring of 2015, in particular Steve Zelditch for his role in organising this program as well as his interest in our previous work which led directly to the material in this note.  We also thank the other participants, in particular Bo Berndtsson for his interest and assistance.

During this work JR was supported by an EPSRC Career Acceleration Fellowship (EP/J002062/1). DWN has received funding from the People Programme (Marie Curie Actions) of the European Union's Seventh Framework Programme (FP7/2007-2013) under REA grant agreement no 329070


\section{The Hele Shaw Flow}
\subsection{Definition and Preliminaries}
Suppose $(X,\omega)$ is a one-dimensional K\"ahler manifold, which we will take to be either $\mathbb P^1$ with its Fubini-Study form $\omega_{FS}$, $\mathbb{C}$ with the Lebesgue form $dA$ or the open unit disc $\mathbb D\subset \mathbb C$ with the Poincar\'e form $\omega_P$. Let $$A:=\int_X \omega \in (0,\infty]$$ and fix an origin $0\in X$.    We use the convention 
$$d^c = \frac{i}{2\pi}(\overline{\partial} - \partial)$$ so $dd^c \log |z|^2 = \delta_0$.  On $\mathbb P^1$ we always having in mind a chart $0\in \mathbb C\subset \mathbb P^1$ with coordinate $z$ so the Fubini-Study metric $\omega_{FS}$ has local potential $\log (1+|z|^2)$ on $\mathbb C$ giving $\mathbb P^1$ area 1.

Let $\phi\in C^{\infty}(X)$ be such that 
\begin{equation}\label{eq:vpositive}
\omega_{\phi} := \omega + dd^c\phi>\epsilon \omega \text{ for some }\epsilon>0.\end{equation}
In particular \eqref{eq:vpositive} implies $\omega_{\phi}$ is strictly positive (and in the compact case $X=\mathbb P^1$ condition \eqref{eq:vpositive} is equivalent to $\omega_{\phi}$ being strictly positive).

\begin{definition}
For $t\in (0,A)$ set
$$\psi_{t} := \sup\{\psi\colon X\to \mathbb R\cup \{-\infty\}: \psi \text{ is $\omega$-psh and }\psi\le \phi \text{ and } \nu_{0}(\psi)\ge t\}.$$
\end{definition}

Here $\nu_0$ denotes the Lelong number at $0$, so $\nu_0(\psi)\ge t$ means that $\psi(z)\le t\ln |z|^2 + O(1)$ near $0$.  As the upper semi-continuous regularisation of $\psi_t$ is itself a candidate for the envelope defining $\psi_t$, we see that $\psi_t$ is $\omega$-psh. Since the Lelong number is additive we get that for a fixed $z$ the functon $t\mapsto \psi_t(z)$ is concave in $t$.

\begin{definition}
For $t\in (0,A)$ set
\begin{equation}
\Omega_t : = \{ z\in X : \psi_t(z)<\phi(z)\}.\label{eq:defHS}
\end{equation} 
\end{definition}

It is easy to see that if $\phi$ is replaced by $\phi+h$ for some harmonic function $h$ then $\psi_t$ is replaced by $\psi_t+h$.  Thus $\Omega_t$ depends only on $\omega_{\phi}$.  

\begin{definition}(Hele-Shaw flow)\label{def:heleshawflow}
We refer to collection of sets $\{\Omega_t\}_{t\in (0,A)}$ as the \emph{Hele-Shaw flow} associated to $(X,\omega_{\phi})$ and the collection $\{\psi_t\}_{t\in (0,A)}$ as the \emph{Hele-Shaw envelopes} associated to $(X,\omega,\phi)$.
\end{definition}

\begin{remark}
What we have called the Hele-Shaw flow is often called the ``weak Hele-Shaw flow".  If $(a,b) \subset (0,A)$ we will also refer to the subcollection $\{\Omega_t\}_{t\in (a,b)}$ as the Hele-Shaw flow and similarly for the envelopes.
\end{remark}

\clearpage

\begin{proposition}\label{prop:basicHS}\ 
\begin{enumerate}
\item \label{item1} $\Omega_t$ is an open connected set containing the origin for all $t\in (0,A)$.
\item \label{item2} $\partial \Omega_t$ has measure zero.
\item \label{item3} $\psi_t$ is $C^{1,1}$ on $X\setminus\{0\}$. 
\item \label{item4} $$\omega_{\psi_t} = (1-\chi_{\Omega_t}) \omega_{\phi} +t\delta_0$$
in the sense of currents. Here $\chi_S$ denotes the characteristic function of a set $S$, and $\delta_0$ the Dirac delta.
\item \label{item5} For $t\in (0, A)$ we have
$$\int_{\Omega_t}\omega_{\phi}=t.$$  
\end{enumerate}
\end{proposition}
\begin{proof}
This is standard material for the Hele-Shaw flow, and the details are given in \cite[Proposition 1.1]{RW} (the cited reference is for $X=\mathbb P^1$, but the same proof applies for $\mathbb D$ or $\mathbb C$).
\end{proof}

Our next Lemma says that the Hele-Shaw flow is local, by which we mean $\Omega_t$ depends only one the restriction of $\omega_{\phi}$ to a neighbourhood of $\overline{\Omega_t}$.  

\begin{lemma}[Locality of the Hele-Shaw Flow]\label{lem:locality}
Let $U\subset X$ be an open subset containing $0$ and $\phi$ and $\tilde{\phi}$ two K\"ahler potentials such that $\omega_{\phi}=\omega_{\tilde{\phi}}$ on $U$. Let $\Omega_t$ and $\tilde{\Omega}_t$ denote their respective Hele-Shaw flows. Then for all $t$ such that $\Omega_t$ is relatively compact in $U$ we have that 
$$\Omega_t = \tilde{\Omega}_t.$$
\end{lemma}
\begin{proof}
If $\Omega_t$ is relatively compact in $U$ it follows that $${\psi_t}_{|U}=\sup\{\psi\leq \phi: \psi \text{ is $\omega$-psh on $U$ and } \nu_0(\psi)\geq t\}.$$ We also have that $\phi-\tilde{\phi}$ is harmonic on $U$, from which it follows that $${(\psi_t-\phi+\tilde{\phi})}_{|U}=\sup\{\psi\leq \tilde{\phi}: \psi \text{ is $\omega$-psh on $U$ and } \nu_0(\psi)\geq t\}$$ and hence 
$$(\psi_t-\phi+\tilde{\phi})_{|U}\geq (\tilde{\psi}_t)_{|U}.$$ This in turn shows that $$\tilde{\Omega}_t \supseteq \Omega_t.$$ But now 

$$t=\int_{\Omega_t}\omega_{\phi}=\int_{\Omega_t}\omega_{\tilde{\phi}}\leq \int_{\tilde{\Omega}_t}\omega_{\tilde{\phi}}=t,$$ hence $$\Omega_t=\tilde{\Omega}_t.$$

\end{proof}


\subsection{The Strong Hele-Shaw flow}

We shall also need the notion of a strong solution to the Hele-Shaw flow.  We shall only consider this in the plane, so suppose $\{\Omega_t\}_{t\in (a,b)}$ is a smooth increasing family of domains of $\mathbb C$.  By this we mean each $\Omega_t$ is smoothly bounded and varies smoothly, so locally $\partial \Omega_t$ is the graph of a smooth function that varies smoothly with $t$.    So if $n$ denotes the outward unit normal vector field on $\partial \Omega_{t_0}$ for some $t_0$, then for $t$ close to $t_0$ we can write $\partial \Omega_t = \{ x + f(x,t) n_x : x\in \partial \Omega_{t_0}\}$ for some smooth function $f_t(x) = f(x,t)$ on $\partial \Omega_{t_0}$ that is positive for $t>t_0$ and negative for $t<t_0$.   Then the \emph{normal velocity} of $\partial \Omega_{t_0}$ is defined to be
$$ V_{t_0} := \frac{df_t}{dt}\big\vert_{t=0} n.$$

Now assume also each $\Omega_t$ contains the origin.    For each $t$ let $$p_t(z):=-G_{\Omega_t}(z,0)$$ where $G_{\Omega_t}$ denotes the Green's function for $\Omega_t$ with logarithmic singularity at the origin. Thus 
$$ p_t = 0 \text{ on } \partial \Omega_t \text{ and } \Delta p_t =-\delta_0.$$
The statement that $p_t$ exists and is smooth on $\overline{\Omega}_t\setminus \{0\}$ is classical (this follows immediately from regularity of the Dirichlet problem for the Laplacian, e.g.\ \cite[Proposition 1.3.11]{Krantz}), which can be found, for instance, in \cite[Chapter 6]{Gilbarg}).     We also fix a smooth area form $\eta$ on $\mathbb C$ which we write as 
$$\eta = \frac{1}{\kappa} dA$$
where $dA$ is the Lebesgue measure and $\kappa$ is a strictly positive real-valued smooth function on $\mathbb C$.

\begin{definition}(Strong Hele-Shaw flow)
We say that $\{\Omega_t\}_{t\in (a,b)}$ is the \emph{strong Hele-Shaw flow} if
\begin{equation}\label{eq:HSclassical:again} 
 V_t = -\kappa \nabla p_t \text{ on } \partial \Omega_t \text{ for } t\in (a,b)
\end{equation}
where $V_t$ is the normal velocity of $\partial \Omega_t$.  When necessary to emphasise the dependence on the area form we refer to this as the strong Hele-Shaw flow with respect to $\eta$ (or with respect to $\kappa$).
\end{definition}

\begin{remark}
 The strong Hele-Shaw flow has an interpretation as the flow of a fluid moving between two plates in a medium which has a permeability encoded by the function $\kappa$, under injection of fluid at the origin (see \cite{RWDisc} for a discussion, and also \cite{Gustafsson} for a comprehensive account of the subject which for the most part considers the case where $\kappa\equiv 1$).   
\end{remark}

We shall now prove that a strong Hele-Shaw flow of simply connected domains is also the Hele-Shaw flow defined using envelopes, as  in Definition \ref{def:heleshawflow}.  To do this we start with the following (slight generalization) of a classical statement due Richardson \cite{Richardson} saying that for the Hele-Shaw flow, the complex moments 
$$M_k(t):= \int_{\Omega_t} z^k \frac{dA}{\kappa} \text{ for } k\in \mathbb Z_{\ge 1}$$
remain constant in $t$.

\begin{lemma}\label{lem:richardson}
Suppose that $\{\Omega_t\}_{t\in (a,b)}$ is a smooth family of strictly increasing simply connected domains in $\mathbb C$ containing the origin that satisfies
\begin{equation}\label{eq:HSclassical:repeat} 
 V_t = -\kappa \nabla p_t \text{ on } \partial \Omega_t.
\end{equation}
Then for any integrable subharmonic function $h$ defined on some neighbourhood of $\overline{\Omega}_t,$ and $t_0<t$ we have
$$  \int_{\Omega_t\setminus \Omega_{t_0}} h \frac{dA}{\kappa}  \ge (t-t_0)  h(0).$$ 
\end{lemma}
\begin{proof}
By the Reynolds' transport theorem and then integration by parts, one computes
\begin{align*}
  \frac{d}{dt} \int_{\Omega_t}h \frac{1}{\kappa} dA &= \int_{\partial \Omega_t} h \frac{V_t}{\kappa} ds = -\int_{\partial \Omega_t} h \frac{\partial p_t}{\partial n} ds \\
&= \int_{\Omega_t} \left(p_t \Delta(h) - h(\Delta p_t)\right) dA - \int_{\partial \Omega_t} p_t \frac{\partial h}{\partial n} ds\ge h(0)
\end{align*}
since $\Delta h\ge 0$ and $p_t=0$ on $\partial \Omega_t$ and $\Delta p_t = -\delta_0$, 
\end{proof}

\begin{corollary}
With the assumption of the above lemma, suppose that $a=0$ and $\Omega_t$ tends to $\{0\}$ as $t\to 0$, that is given any neighbourhood $U$ of the origin $\Omega_t\subset U$ for $t$ sufficiently small.  Then for any integrable subharmonic function $h$ on $\Omega_t,$  we have
$$  \int_{\Omega_t} h \frac{dA}{\kappa}  \ge t  h(0).$$ 
and moreover equality holds if $h$ is holomorphic.
\end{corollary}
\begin{proof}
Taking the limit as $t_0\to 0$ in the above Lemma gives the first statement.  The second follows as if $h$ is holomorphic then $h$ and $-h$ are subharmonic.
\end{proof}

\begin{remark}In particular taking $h(z) = z^k$ we deduce that the complex moments of a Hele-Shaw flow $\{\Omega_t\}$ tending to $\{0\}$ as $t$ tends to zero are
$$ M_k(t) = \int_{\Omega_t} 
z^k \frac{dA}{\kappa} = 0 \text{ for all } t>0.$$
\end{remark}

We apply this in the next statement to prove that the strong Hele-Shaw flow is also a weak one (with respect to a suitable potential).

\begin{proposition}(Gustafsson)\label{prop:gust}
Suppose that $\{\Omega_t\}_{t\in (0,b)}$ is a smooth family of strictly increasing simply connected domains that is the strong Hele-Shaw flow with respect to $\kappa$, and assume $\{\Omega_t\}_{t\in (0,b)}$ tends to $\{0\}$ as $t\to 0$.   Set
$$ \phi(z) = \int_{\mathbb C} \log |z-\zeta|^2 \frac{dA_\zeta}{\kappa(\zeta)} - \log(1 + |z|^2) \text{ for } z\in \mathbb C.$$
Then $\{\Omega_t\}_{t\in (0,b)}$ is the Hele-Shaw flow with respect to $\omega_\phi : = dd^c ( \log( 1+ |z|^2) + \phi) $.
\end{proposition}
\begin{proof}
For the proof we shall write $\Omega_t^w:= \{ z\in X : \psi_t(z)< \phi(z)\}$ for the Hele-Shaw flow with respect to $\omega_{\phi}$, so the goal is to prove $\Omega^w_t = \Omega_t$.   
Define
$$ \tilde{\psi_t}(z) := \int_{\Omega_t^c}  \log |z-\zeta|^2 \frac{dA_\zeta}{\kappa(\zeta)} - \log(1 + |z|^2) + t \ln |z|^2.$$
Then by construction $\omega_{\tilde{\psi}_t}\ge 0$ and $\nu_0(\tilde{\psi}_t) = t$.   As $h(\zeta):= \log |z-\zeta|^2$ is subharmonic and integrable, we get from the previous Corollary that for all $z\in \mathbb C$
\begin{equation}
\phi(z) - \tilde{\psi}_t(z)= \int_{\Omega_t}  \log |z-\zeta|^2 \frac{dA_\zeta}{\kappa(\zeta)}  - t \ln |z|^2\ge 0.\label{eq:richapplication}
\end{equation}
Hence $\tilde{\psi_t}\le \phi$ making it a candidate for the envelope defining $\psi_t$, and hence $\tilde{\psi}_t\le \psi_t$.  In fact more is true, and if $z\notin \Omega_t$ then $h$ is holomorphic on $\Omega_t$ so equality holds in \eqref{eq:richapplication}, and hence
$$ \tilde{\psi}_t = \psi_t = \phi \text{ on } \Omega_t^c.$$
Now $\tilde{\psi_t} + \log (1 + |z|^2)$  and $\psi_t + \log( 1+ |z|^2)$ are both harmonic on $\Omega_t\setminus\{0\}$ (by Proposition \ref{prop:basicHS}(4)) with Lelong number one at $0$.  Hence by the maximum principle $\tilde{\psi}_t = \psi_t$ on $\Omega_t\setminus \{0\}$ as well.  Thus we conclude $\Omega_t^w = \Omega_t$ as desired.  \end{proof}

\begin{remark}\label{rmk:strongconverse}
Although we will shall not really need it, we remark that there is a converse to this, which says that if the Hele-Shaw domain $\Omega_T$ (with respect to $\omega+dd^c\phi$) is a smoothly bounded Jordan domain for some $T\in (0,V)$ then there is an $\epsilon>0$ such that $\{\Omega_t\}_{t\in(T-\epsilon,T+\epsilon)}$ is actually the strong Hele-Shaw flow.   Thus the hypothesis that $\{\Omega_t\}$ varies smoothly in Theorem \ref{thm:main} (as well as in Definition \ref{def:tangency}) is redundant.     The proof of this statement follows easily from the work in \cite{RWDisc}; specifically from \cite[Remark 3.12]{RWDisc} the Hele-Shaw domains $\Omega_t$ all lift to holomorphic curves $\Sigma_T$ in $\mathbb C\times \mathbb P^1$ with boundary contained in the submanifold given as the graph of $\frac{\partial \phi}{\partial z}$.  The hypothesis on $\Omega_T$ imply that $\Sigma_T$ is a holomorphic disc, at which point we can run the proof of \cite[Theorem 2.2]{RWDisc}.
\end{remark}

Finally, we state two previous results of the authors that give existence results for the strong Hele-Shaw flow.  The first says this flow always exists for small time.  

\begin{theorem} \label{thm:RWNold0}\cite[Theorem 2.1]{RWDisc}
The Hele-Shaw flow for any K\"ahler form $\omega+dd^c\phi$ is the strong Hele-Shaw flow for short time $t\in(0,\epsilon),$ $\epsilon>0,$ and in this range is diffeomorphic to the standard flow $B(t)$.
\end{theorem}

The second says that any simply connected bounded Jordan domain $\Omega$ is part of a strong Hele-Shaw flow, both backwards and forwards in time, as long as one allows a modification of the area form inside $\Omega$.

\begin{theorem} \label{thm:RWNold2}\cite[Theorem 2.2, Remark 7.1]{RWDisc}
Let $\Omega$ be a smoothly bounded Jordan domain in $\mathbb C$ containing the origin and let $\eta$ be a smooth area form defined in a neighbourhood of $\partial\Omega$. Then there exists a smooth area form $\eta'$ on a neighbourhood $U$ of $\partial \Omega$ such that $\eta=\eta'$ on $U\cap \Omega^c$ and so that $\Omega=\Omega_T$ is part of a strong Hele-Shaw flow $\{\Omega_t\}_{t\in(T-\epsilon,T+\epsilon)}$ with respect to $\eta'$.  
\end{theorem}

\section{Designer Potentials}\label{sec:designer}

In this section we show how to produce potentials with particular prescribed properties (we do this only on $\mathbb P^1$ but a similar story holds for $\mathbb D$).  We first show that any (reasonable) strictly increasing family of smooth domains in $\mathbb P^1$ is the Hele-Shaw flow for some smooth K\"ahler potential.   Recall $B(t)$ denotes the geodesic ball centred at the origin of area $t$ taken with respect to $\omega_{FS}$.

\begin{theorem}\label{thm:existencepotential}
Suppose $ \{\Omega_t\}_{t\in (0,1)}$ is a family of subsets of $\mathbb P^1$ that is 
\begin{enumerate}
\item smoothly bounded, varies smoothly,  and is simply connected for all $t$,
\item strictly increasing, i.e.\ $\Omega_t \Subset \Omega_{t'}$ for $t<t'$, with non-vanishing normal velocity of the boundary $\partial \Omega_t,$ and
\item standard as $t$ tends to 0 and as $t$ tends to 1.
\end{enumerate}
 Then there exists a smooth $\phi\in C^{\infty}(\mathbb P^1)$ such that $\{\Omega_t\}_{t\in (0,1)}$ is the Hele-Shaw flow with respect to the K\"ahler form $\omega_{FS} + dd^c\phi$.
\end{theorem}

\begin{proof}
The idea of the proof is to construct a smooth function $\kappa$ on $\mathbb C$ such that $\{\Omega_t\}$ is the strong Hele-Shaw flow with respect to the permeability $\kappa$.   Since $\{\Omega_t\}_{t\in(0,1)}$ is assumed to be standard as $t$ tends to 1 we have $\Omega_t\subset \mathbb C$ for all $t\in (0,1)$ and so by Lemma \ref{lem:locality} we may as well consider the Hele-Shaw flow as taking place in $\mathbb C$.  Let $p_t$ satisfy
$$ p_t = 0 \text{ on } \partial \Omega_t \text{ and } \Delta p_t =-\delta_0.$$
As already  mentioned, the fact that $p_t$ exists and is smooth on $\overline{\Omega}_t\setminus \{0\}$ is classical.   What is also true is that $p_t$ varies smoothly with $t$; this is presumably also well-known in some circles, but since we were not able to find a convenient reference we give a proof in the Appendix (Corollary \ref{cor:ptsmoothint}).

Assuming this smoothness for now,  we use $p_t$ to define a function $\kappa$ by requiring that
\begin{equation}
 V_t = -\kappa \nabla p_t \text{ on } \partial \Omega_t\text{ for } t\in (0,1)\label{eq:HSclassical}
\end{equation}
where $V_t$ is the normal velocity of $\partial \Omega_t$.   Since $\{\Omega_t\}_{t\in (0,1)}$ is increasing smoothly and $V_t$ was assumed to be non-vanishing we see that $\kappa$ is a well-defined strictly positive smooth function on $\mathbb C\setminus \{0\}$.  

Now we use the assumption that $\{\Omega_t\}_{t\in (0,1)}$ is standard as $t$ tends to zero to deduce that $\kappa$ extends to a smooth function over $0$.  Assume $t\ll 1$.  By explicit calculation with the Fubini-Study metric we know that the disc with area $t$ has radius
$$ R_t: = \left(\frac{t}{1-t}\right)^{1/2}.$$
To see this, recall our convention with the Fubini-study metric is that $\mathbb P^1$ has area 1, and the formula for $R_t$ follows from the calculation
$$\int_{|z|<R}  \frac{dxdy}{2\pi(1+|z|^2)^2} = \int_{0}^R \frac{r dr}{(1+r^2)^2} = 1- \frac{1}{1+R^2}.$$
Thus by symmetry, for sufficiently small $t$
$$ \Omega_t = \{ z\in \mathbb C : |z|<R_t\}$$
and so
 $$p_t(z)= -\frac{1}{4\pi} ( \log |z|^2  - \log(R_t^2)).$$   Clearly $\kappa$ is radially symmetric near $0$, so it is sufficient to compute it at a point $z_t:=(R_t,0)$ for small $t$.  To do so observe that at $z_t$ we have $$\nabla p_t = -\frac{1}{2\pi R_t} \left( \begin{array}{c} 1 \\ 0 \end{array}\right).$$  On the other hand the normal velocity of $\partial \Omega_t$ at the point $z_t$ is $\frac{dR_t}{dt}\left(\begin{smallmatrix}1\\0\end{smallmatrix}\right)$ and so the defining equation \eqref{eq:HSclassical} for $\kappa$ becomes
$$\frac{1}{2 R_t}\frac{1}{(1-t)^2} =  \frac{\kappa(z_t)}{2\pi R_t}.$$
After some calculation this yields
\begin{equation}
 \kappa(z)  = \pi (1+ |z|^2)^2 \text{ near } z=0 \label{eq:kappanear0}
 \end{equation}
which clearly extends smoothly over $z=0$. 

Now define
\begin{equation}
 \phi(z) =  \int_{\mathbb C} \log |z-\zeta|^2 \frac{dA_\zeta}{\kappa(\zeta)} - \log(1 + |z|^2) \text{ for } z\in \mathbb C\label{eq:defdesignerphi}
\end{equation}
which is a smooth function on $\mathbb C$ chosen so that
\begin{equation}
 dd^c (\log(1+|z|^2) + \phi) = \frac{dA}{\kappa} \text{ on } \mathbb C.\label{eq:defphi}
\end{equation}
    
Using that $\{\Omega_t\}_{t\in (0,1)}$ is standard as $t$ tends to infinity we have that \eqref{eq:kappanear0} also holds for $|z|$ sufficiently large.  We claim this implies $\phi$ extends to a smooth function on $\mathbb P^1$ and $\omega_{\phi}$ is strictly positive on $\mathbb P^1$.   To see this, start with the identity
    $$\frac{1}{\pi}\int_{\mathbb C} \frac{ \log |z-\zeta|^2}{(1+ |\zeta|^2)^2} dA_\zeta = \log( 1+ |z|^2)$$
 (this can be seen by noting that the difference is harmonic on $\mathbb C$ bounded and equal to zero at $z=0$).  Now the same calculation as above means the assumption that $\{\Omega_t\}$ is standard as $t$ tends to 1 implies $C>0$ such that \eqref{eq:kappanear0} holds on $\{ |z|>C\}$.  Therefore
 $$\phi(z) = \frac{1}{\pi} \int_{|z|<C} \log |z-\zeta|^2\left( \frac{\pi}{\kappa(\zeta)} - \frac{1}{(1+ |\zeta|^2)^2}\right) dA_\zeta$$
which one sees extends smoothly over $z=\infty$ in such a way that makes $\omega_{\phi}$ strictly positive as claimed.  (We remark that this can also be seen abstractly, since the flow being standard near $0$ and $\infty$ means that $\kappa$ has to agree with the permeability for the standard flow for $(\mathbb P^1,\omega_{FS})$.)

Now by construction $\{\Omega_t\}$ is the strong Hele-Shaw flow with respect to $\kappa$, and hence by Proposition \ref{prop:gust}, is the strong Hele-Shaw flow for $\omega_{FS} + dd^c\phi$ as desired. 
\end{proof}

\begin{remark} \label{remark:thm:existencepotential}
We observe that the above proof actually shows slightly more, namely that if $\{\Omega_t\}_{t\in (0,T]}$ is a smooth family of strictly increasing domains that is standard as $t\to 0$ then setting $X':=\Omega_{T}$ there exists a $\phi\in C^{\infty}(X')$ such that $\{\Omega_t\}_{t\in (0,T)}$ is the Hele-Shaw flow for $(X',\omega_{\phi})$.
\end{remark}

\begin{remark}
The Hele-Shaw flow depends only on the form $\omega + dd^c\phi$.  From the proof of Theorem \ref{thm:existencepotential} one sees that $\{\Omega_t\}$ determines $\kappa$ uniquely, and thus $\phi$ is unique up to addition of a harmonic function.
\end{remark}

Now let $S$ be a finite union of points and non-intersecting smooth embedded curve segments in $\mathbb{P}^1\setminus \{0\}.$ Using similar ideas to above we now show that there are K\"ahler potentials whose Hele-Shaw flow is smoothly bounded and simply connected until it develops a tangency along $S$.

\begin{proposition}\label{prop:existencetangency}
There exists a $\phi\in C^{\infty}(\mathbb P^1)$ such that $\omega_{\phi}$ is strictly positive, and whose associated Hele-Shaw flow develops tangency along $S$.
\end{proposition}

\begin{proof}
It is clear that one can find a simply connected domain $\Omega$ containing $0$ such that $\partial \Omega$ is the image of a smooth locally embedded curve $\gamma$ intersecting itself tangentially precisely along $S$ and so $\overline{\Omega_t}\setminus S$ is connected as in Figure \ref{fig1} (use induction on the number of components of $S$). Let $$T:=\int_{\Omega}\omega_{FS}
.$$ 
We construct the Hele-Shaw flow backwards starting at $\Omega_{T}:=\Omega$.

Pick a point $z_i$ in each connected component of $\mathbb{P}^1\setminus \overline{\Omega_T},$ and let $\pi$ be the projection from the universal cover $\Sigma$ of $\mathbb{P}^1$ with the points $z_i$ removed. Then $\gamma$ lifts to a smooth embedded curve in $\Sigma$ and so $\pi^{-1}(\Omega_T)$ is a disjoint union of copies of $\Omega_T$.   We pick one of them and call it $\Omega'$ which is smoothly bounded and simply connected.   Then Theorem \ref{thm:RWNold2} implies that there exists a smooth area form $\eta'$ on a neighbourhood of $\Sigma\setminus \Omega',$ equal to $\eta:=\pi^*\omega_{FS}$ on $\Sigma\setminus \Omega',$ such that the strong Hele-Shaw flow exists starting from $\Omega'$ with respect to $\eta'$ for a short while backwards in time.   We denote the projection of this Hele-Shaw flow to $\mathbb P^1$ by $\{\Omega_t\}_{t\in(T-\epsilon,T]}$. 

We then extend this to a family of domains $\Omega_t,$ $t\in(0,T),$ in $\mathbb{P}^1$, with the properties as in Theorem \ref{thm:existencepotential}, so by this Theorem and Remark \ref{remark:thm:existencepotential} we have an area form $\omega'$ on $\Omega_T$ such that $\{\Omega_t\}_{t\in(0,T)}$ is a strong Hele-Shaw flow with respect to $\omega'$. We also have that $\omega'=\omega_{FS}$ on $\Omega_{T}\setminus \Omega_{T-\epsilon}$. We can thus extend $\omega'$ to a smooth K\"ahler form on $\mathbb{P}^1$ by letting it be equal to $\omega_{FS}$ on $\mathbb{P}^1\setminus \Omega_{T}$.   Thus $\{\Omega_t\}_{t\in (0,T)}$ is the strong Hele-Shaw flow with respect to the area form $\omega'$ on $\mathbb P^1$, and thus also the Hele-Shaw flow by Proposition \ref{prop:gust}.  

On the other hand, by the continuity of the Hele-Shaw flow (applied on $\Sigma$) it follows that $\Omega_T$ is the Hele-Shaw domain of $\omega'$ at time $T$. Thus if $\phi$ is a smooth function so that $\omega'=\omega_{FS}+dd^c\phi$ we get that the Hele-Shaw flow with respect to $\phi$ develops a tangency along $S$ at time $T.$
\end{proof}

\begin{remark}
If we assume in addition that $S$ is such that one can find such an $\Omega_T$ with real-analytic boundary, then instead of using Theorem \ref{thm:RWNold2}  one can use the classical short-time existence result of the Hele-Shaw backwards for small time,  starting with simply connected domain with real analytic boundary.   Such $S$ do give explicit singularities of geodesic rays at specific points (Theorem \ref{thm:notc2}), but the assumption that $\Omega_T$ need have real analytic boundary strictly decreases the collection of $S$ to which the theorem applies.
\end{remark}

\section{Dirichlet Problem for the Homogeneous Monge-Amp\`ere Equation}\label{sec:HCMA}

In this section and the next we will mainly focus on the case $(X,\omega)=(\mathbb{P}^1,\omega_{FS})$.  We will return to the case $(\mathbb{D}, \omega_P)$ at the end of Section \ref{sec:reg}.

\subsection{Preliminary definitions}
We first consider two versions of the Dirichlet Problem for the Homogeneous complex Monge-Amp\`ere Equation, first over the disc and second over the punctured disc.    Again we let $\phi\in C^{\infty}(\mathbb{P}^1)$ be such that $\omega_{FS}+ dd^c\phi>0$, and $\pi_{\mathbb{P}^1}\colon \mathbb{P}^1\times \overline{\mathbb D} \to \mathbb{P}^1$ and $\pi_{\mathbb D}\colon \mathbb{P}^1\times \overline{\mathbb D} \to \overline{\mathbb D}$ be the projections. 

\begin{definition}(Weak Solution to the HCMA)\

\begin{enumerate}
\item Let
\begin{equation}\label{eq:weaksolutionD}
 \Phi := \sup 
 \left\{
\begin{array}{c}
  \psi \colon  \mathbb{P}^1\times \overline{\mathbb D}\to \mathbb R\cup \{-\infty\} : \psi \text{ is usc, }\pi_{\mathbb{P}^1}^*\omega + dd^c\psi \ge 0 \\\text{ and } \psi(z,\tau)\le \phi(\tau z) \text{ for } (z,\tau)\in \mathbb{P}^1\times  \partial\mathbb D
 \end{array}
 \right\}.
\end{equation}
\item Let
\begin{equation}\label{eq:weaksolutionDtimes}
\tilde{\Phi} :=\sup \left\{
\begin{array}{c}
\psi \colon \mathbb{P}^1\times \overline{\mathbb D}\to \mathbb R\cup \{-\infty\} : \psi\text{ is usc, } \pi_{\mathbb{P}^1}^*\omega + dd^c\psi\ge 0 \\\text{and } \psi(z,\tau)\le \phi(z) \text{ for } (z,\tau)\in \mathbb{P}^1\times \partial \mathbb D \text{ and } \nu_{(0,0)}(\psi)\ge 1
\end{array}
\right\}.
\end{equation}
\end{enumerate}
\end{definition}

So the difference between these two definitions is that in the second the boundary data is $S^1$-invariant but has an additional requirement of giving a prescribed singularity at the point $(0,0)$.  However these two quantities carry the same information as given by:
\begin{proposition} \label{prop:equiv}
We have that $$\Phi(z,\tau)+\ln|\tau|^2+\ln(1+|z|^2)=\tilde{\Phi}(\tau z,\tau)+\ln(1+|\tau z|^2) \text{ for } (z,\tau)\in \mathbb P^1\times \overline{\mathbb D}^{\times}.$$
\end{proposition}
\begin{proof}
This is proved in \cite[Proposition 2.3]{RW}.
\end{proof}

\begin{definition}(Regular solution)
We say that $\Phi$ is \emph{regular} on an open subset $S\subset \mathbb{P}^1\times \overline{\mathbb D}$ if it is smooth on $S$ and the restriction of $\pi_{\mathbb{P}^1}^* \omega + dd^c \Phi$ to $S_{\tau}:= \pi_{\mathbb D}^{-1}(\tau)\cap S$  is strictly positive for all $\tau\in \mathbb D$.  Similarly we say $\tilde{\Phi}$ is \emph{regular} on $S$ if it is smooth on $S\setminus \{(0,0)\}$ and the restriction of $\pi_{\mathbb{P}^1}^* \omega + dd^c \tilde{\Phi}$ to $S_{\tau}$ is strictly positive for all $\tau\in \overline{\mathbb D}^{\times}$.

Finally we say that $\Phi$ (resp.\ $\tilde{\Phi}$) is \emph{regular} if it is regular on all of $\mathbb{P}^1\times \overline{\mathbb D}$ (resp. $X\times \overline{\mathbb D}^\times)$.
\end{definition}

By well-known arguments (see \cite{BedfordTaylor}),  $\tilde{\Phi}$ is usc,  $\pi_{\mathbb{P}^1}^*\omega + dd^c\tilde{\Phi}\ge 0$ and $(\pi_{\mathbb{P}^1}^*\omega + dd^c\tilde{\Phi})^2=0$ away from $(0,0)$ and $\tilde{\Phi}(z,\tau) = \phi(z)$ for $\tau\in \partial \mathbb D$.   Moreover it is not hard to show that $\tilde{\Phi}$ is locally bounded away from $(0,0)$ and $\nu_{(0,0)}\tilde{\Phi}=1$.     Thus $\tilde{\Phi}$ is the weak solution to Dirichlet problem to the Homogeneous Monge-Amp\`ere Equation with boundary data consisting of $\phi(z)$ on $\mathbb{P}^1\times \partial \mathbb D$, and this prescribed singularity at $(0,0)$.     Thinking of $s:= -\ln |\tau|^2$ for $\tau\in \mathbb D^{\times}$ as a time variable let $\phi_s(\cdot) = \tilde{\Phi}(\cdot,\tau)$.  Then the map 
$$s\mapsto \phi_s\text{ for } s\in [0,\infty)$$ is a weak geodesic ray in the space of weak K\"ahler potentials that starts with $\phi$ and has limit the singular potential $\ln |z|^2$ as $s$ tends to infinity.  Moreover if $\tilde{\Phi}$ is regular this is a smooth geodesic ray in the space of K\"ahler metrics.

Similarly $\Phi$ is the weak solution to the same Dirichlet problem  over $\mathbb{P}^1\times \overline{\mathbb D}$ with prescribed boundary $\phi(\tau z)$ over $\mathbb{P}^1\times \partial \mathbb D$.

\subsection{The Duality Theorem}

The duality between $\tilde{\Phi}$ and the Hele-Shaw envelopes $\psi_t$ is provided by the following:

\begin{theorem}(Ross--Witt Nystr\"om \cite[Theorem 2.7]{RW})\label{thm:Legendre} 
Let $\psi_t$ be the Hele-Shaw envelopes associated to $(\mathbb{P}^1,\omega_{FS},\phi)$ and $\tilde{\Phi}$ be the weak solution to the Homogeneous Monge-Amp\`ere Equation as defined in \eqref{eq:weaksolutionDtimes}.   Then
\begin{equation} \label{legendre1}
\psi_t(z)=\inf_{|\tau|>0}\{\tilde{\Phi}(z,\tau)-(1-t)\ln|\tau|^2\}
\end{equation} 
and 
\begin{equation} \label{legendre2}
\tilde{\Phi}(z,\tau)=\sup_{t}\{\psi_{t}(z)+(1-t)\ln |\tau|^2\}.
\end{equation}
\end{theorem}

\section{Regularity of Geodesic Rays} \label{sec:reg}

We continue with the notation from the previous section, so $\tilde{\Phi}$ is as defined in \eqref{eq:weaksolutionDtimes}.  Since $\tilde{\Phi}(z,\tau)$ is $\pi_{\mathbb{P}^1}^*\omega_{FS}$-psh and and independent of the argument of $\tau$ it follows that for a fixed $z$ the map  $s\mapsto \tilde{\Phi}(z,e^{-s/2})$ is convex in $s:=-\ln |\tau|^2$. Hence the right derivative $$\frac{\partial}{\partial s +} \tilde{\Phi}(z,e^{-s/2})$$ always exists.

\begin{definition}\label{def:ham}
 We define
$$H(z,\tau): = \frac{\partial}{\partial s^+} \tilde{\Phi}(z,e^{-s/2}) \text{ for } (z,\tau)\in X\times \overline{\mathbb D}^{\times}$$
where $s:=-\ln |\tau|^2$.  
\end{definition}

\begin{remark}
By a result of Chen \cite{Chen}, with complements by B\l ocki \cite{Blocki2}, the function $\tilde{\Phi}$ is in fact $C^{1,1}$ and thus $H$ is continuous (even Lipschitz but we will not use this).   
\end{remark}

A key connection with the Hele-Shaw flow is given by:

\begin{proposition}\label{prop:ham}
$$H(z,1) + 1 = \sup\{ t : \psi_t(z) = \phi(z) \} = \sup\{ t : z\notin \Omega_t\}.$$
\end{proposition}
\begin{proof}
This is \cite[Proposition 2.8]{RW} and for convenience we repeat the proof here.    From \eqref{legendre2} if $\psi_{t}(z)=\phi(z)$ then $$\tilde{\Phi}(z,e^{-s/2})\geq (t-1)s+\phi(z)$$ and thus $$H(z,1)\geq\sup\{t: \psi_{t}(z)=\phi(z)\}-1.$$    Suppose $\psi_{t}(z)\le \phi(z) +a$ for some $a<0$.  For a fixed $z$ the function $t'\mapsto \psi_{t'}(z)$ is concave and decreasing in $t'$ \cite[Prop. 1.3]{RW}, so for $t\le t'<A$ and $s\ge 0$ we have $\psi_{t'}(z) + (t'-1)s \le \phi(z) +a$.  On the other hand we always have $\psi_{t'}\le \phi$ so if $0\le t'\le t$ then $\psi_{t'}(z) + (t'-1)s\le \phi(z) + (t-1)s$.  Putting this together with (\ref{legendre2}) gives $$\tilde{\Phi}(z,e^{-s/2})\le \phi(z) +  \max( (t-1)s,a)$$ and so $H(z,1)\leq t-1,$ which proves the proposition.
\end{proof}

When the Hele-Shaw domains $\Omega_t$ are simply connected one can say even more.

\begin{definition*}
   Let $f\colon \mathbb D\to \mathbb{P}^1$ be holomorphic.  We say that the graph of $f$ is a  \emph{harmonic disc} for $\Phi$ if $\Phi$ is $\pi_X^*\omega$-harmonic along the graph of $f$, i.e the restriction of $\pi_X^*\omega+dd^c\Phi$ to 
   $\{ (f(\tau), \tau) : \tau\in \mathbb D\}$
vanishes.
\end{definition*}

The main result in \cite{RW} is the following, which says that one can characterize all harmonic discs of $\Phi$ in terms of simply connected Hele-Shaw domains.

\begin{theorem} \label{RWmain}
The graph of $f\colon \mathbb D\to X$ is a harmonic disc for $\Phi$ iff either (1) $f\equiv 0$ or (2) $f(\tau)=\tau^{-1}z$ where $z\in \Omega_1^c$ or (3) $\tau \mapsto \tau f(\tau)$ is a Riemann mapping to a simply connected Hele-Shaw domain $\Omega_t$ taking $0\in \mathbb D$ to $0\in \Omega_t$. 

The function $H$ is constant along the associated discs $\{(\tau f(\tau),\tau)\}$, in the first case $H=-1$, in the second case $H=0$ while in the third case $H=t-1$. 
\end{theorem}

We are now ready to prove Theorem \ref{thm:main} in the case of $(X,\omega)=(\mathbb{P}^1,\omega_{FS})$:

\begin{theorem}\label{thm:main:repeat}
  Suppose the flow Hele-Shaw $\{\Omega_t\}_{t\in [0,1]}$  for a K\"ahler form $\omega_{FS} + dd^c\phi$ on $\mathbb{P}^1$ satisfies
\begin{enumerate}
\item $\{\Omega_t\}_{t\in (0,1)}$ is smoothly bounded and varies smoothly with non-vanishing normal velocity,
\item $\Omega_t$ is simply connected for all $t\in (0,1)$,
\item $\{\Omega_t\}_{t\in (0,1)}$ is standard as $t$ tends to $1$.
\end{enumerate}
Then the weak geodesic ray \eqref{legendre2} obtained as the Legendre transform of the Hele-Shaw envelopes  $\{\psi_t\}$  is regular, and so defines a smooth geodesic ray in the space of K\"ahler metrics on $X$.
\end{theorem}

\begin{proof}

Since by assumption $\Omega_t$ is simply connected for all $t\in (0,1)$ it follows from Theorem \ref{RWmain} that if $\tau \mapsto \tau f(\tau)$ is a Riemann mapping to some $\Omega_t$, $t\in(0,1)$ such that $0\in \mathbb D$ maps to $0\in \Omega_t$, then the graph of $f$ is a harmonic discs of $\Phi$. Note that for fixed $e^{i\theta}\in S^1$, if $\tau \mapsto \tau f(\tau)$ is such a Riemann mapping then so is $\tau \mapsto e^{i\theta}\tau f(e^{i\theta}\tau)$.  Thus for each $t\in(0,1)$ there is an $S^1$-family of corresponding harmonic discs. We also have from Theorem \ref{RWmain} that the graph of $f\equiv 0$ is a harmonic disc, and since $\{\Omega_t\}_{t\in (0,1)}$ is standard as $t$ tends to 1, also the graph of $f\equiv \infty$ is a harmonic disc. 

That these harmonic discs do not intersect for different values of $t$ is clear as they correspond to different values of $H$ and it is also easy to see that the union of all these discs cover $\mathbb{P}^1\times\overline{\mathbb D}$.

Now by Theorem \ref{thm:RWNold0} the foliation is diffeomorphic to the product foliation in a neighbourhood of $\{0\}\times \overline{\mathbb D}$, so in particular it is smooth. The assumption that the Hele-Shaw flow is standard as $t$ tends to $1$ ensures the foliation is also smooth near $\{\infty\}\times \overline{\mathbb D}$.   

Since $\Omega_t$ varies smoothly in $t$, we can find Riemann maps $f_t:\mathbb D \to \Omega_t$ that satisfies $f_t(0)=0$ that vary smoothly with $t$ (see Corollary \ref{cor:riemannmapssmoothint} in the Appendix).  Hence the harmonic discs give a smooth foliation in the remaining part of $\mathbb{P}^1\times \overline{\mathbb D}$, since the normal velocity of $\{\Omega_t\}$ is assumed to be non-vanishing, so every $z\in \mathbb{P}^1\setminus\{0\}$ lies in the boundary of precisely one $\Omega_t$.

Let $D = \{ (f(\tau), \tau)\}$ be one of the harmonic discs. Then $\Phi(f(\tau),\tau)$ is harmonic along $D$, thus for any point $\tau$,  $\Phi(f(\tau),\tau)$ can be expressed as an integral of $\Phi$ over $\partial D$, and the integral depends smoothly on $\tau$.  But $\Phi=\phi$ over $\partial D$ (which is smooth) and the foliation varies smoothly, from which we conclude that $\Phi$ must in fact be smooth.  

For the regularity we argue as follows.  For $\tau\neq 0$ let $T_{\tau}\colon \pi_{\mathbb D}^{-1}(1)\ \to \pi_{\mathbb D}^{-1}(\tau)$ be the flow along the leaves of the above foliation and set $\sigma_{\tau} : = \pi_{\mathbb{P}^1}^* \omega_{FS} + dd^c \Phi |_{\pi_{\mathbb D}^{-1}(\tau)}$.  Then by what is now considered a classical calculation (originally due to Semmes \cite{Semmes} and Donaldson \cite{Donaldson}, see also \cite[Proposition 3.4]{RossNystromTubular} and \cite[Sec 3]{RubinsteinZelditchIII}) we know $T_{\tau}^* \sigma_{\tau} = \sigma_1$.  But $\sigma_1=\omega_{\phi}$ is certainly strictly positive, and hence $\sigma_{\tau}$ is strictly positive as well.

Thus we see that $\Phi$ is a regular solution to the HCMA, and so $\phi_s(\cdot) = \tilde{\Phi}(\cdot,\tau)$ is a regular geodesic ray.

\end{proof}

\begin{remark}
The geodesic ray produced in Theorem \ref{thm:main:repeat} is the same as an example given by Donaldson \cite[p. 24]{Don3}.  The point of view taken there is slightly different, and the initial data is to consider $\mathbb P^1 = S^2 = \{ (x,y,z) : x^2+y^2+z^2=1\}$ and let $h:\mathbb P^1\to \mathbb R$ be smooth, such that $h(x,y,z) = z$ near the poles $z=\pm 1$, and so that $h$ has no further critical points.  Then the sublevel sets $\Omega_t:=\{ f(z)\le t\}$ are all discs, and Donaldson uses the associated Riemann maps to describe explicitly a smooth geodesic ray in the space of K\"ahler potentials on $\mathbb P^1$.
\end{remark}

Now we return to the case $(X,\omega)=(\mathbb{D},\omega_P)$, with $\phi\in C^{\infty}(\mathbb{D})$ such that $\omega_P+dd^c\phi>\epsilon \omega_P$ for some $\epsilon>0$.

Define a function $\tilde{\Phi}(z,\tau)$ on $\mathbb{D}\times \overline{\mathbb{D}}^{\times}$ by
\begin{equation} \label{legendre20}
\tilde{\Phi}(z,\tau)=\sup_{t\in [0,\infty)}\{\psi_{t}(z)-t\ln |\tau|^2\}.
\end{equation}

For a fixed $z$ the function $z\mapsto \tilde{\Phi}(z,e^{-r/2})$ is convex in $r$ which allows us to define the Hamiltonian function:

\begin{definition}\label{def:ham2}
Set
$$H(z,\tau): = \frac{\partial}{\partial r^+} \tilde{\Phi}(z,e^{-r/2}) \text{ for } (z,\tau)\in \mathbb{D}\times \overline{\mathbb D}^{\times}$$
where $r:=-\ln |\tau|^2$. We also let $$H_-(z,\tau): = \frac{\partial}{\partial r^-} \tilde{\Phi}(z,e^{-r/2}).$$  
\end{definition}

Theorem \ref{thm:main} in the case $(X,\omega)=(\mathbb{D},\omega_P)$ can now be formulated in the following way:

\begin{theorem}
 Suppose the flow Hele-Shaw $\{\Omega_t\}_{t\in [0,\infty)}$  for $\omega_P + dd^c\phi>\epsilon \omega_P$ ($\epsilon>0$) on $\mathbb{D}$ satisfies
\begin{enumerate}
\item $\{\Omega_t\}_{t\in (0,\infty)}$ is smoothly bounded and varies smoothly with non-vanishing normal velocity,
\item $\Omega_t$ is simply connected for all $t\in (0,\infty)$.
\end{enumerate}
Then the function $\tilde{\Phi}(z,\tau)$ is a regular solution to the HCMA on $\mathbb{D}\times \mathbb{D}^{\times}$.
\end{theorem}

\begin{proof}
Pick $t,T$ such that $0<t<T<\infty$. It is clear that one can find an increasing family of domains $\Omega_s'$ in $\mathbb{P}^1$ such that 
\begin{enumerate}
\item $\{\Omega_s'\}_{s\in (0,1)}$ is smoothly bounded and varies smoothly with non-vanishing normal velocity,
\item $\Omega_s'$ is simply connected for all $s\in (0,1)$,
\item $\{\Omega_s'\}_{s\in (0,1)}$ is standard as $s$ tends to $1$,
\item $\Omega_s'=\Omega_{Ts}$ for $s\in (0,t/T)$.
\end{enumerate}

By Theorem \ref{thm:existencepotential} $\{\Omega_s'\}_{s\in (0,1)}$ will be the Hele-Shaw flow of some K\"ahler form $\omega_{FS}+dd^c\phi'$ on $\mathbb{P}^1$ and by Theorem \ref{thm:main:repeat} the associated weak geodesic ray $\tilde{\Phi}'$ will be regular.  Moreover by the construction of $\phi'$ is is clear that $$\omega_P+dd^c\phi=T(\omega_{FS}+dd^c\phi')$$ on $\Omega_t$. 

Let $u_P$ be a smooth function on $\mathbb{D}$ such that $\omega_P=dd^cu_p$ and $u_{FS}$ a smooth function on $\mathbb{C}$ such that $\omega_{FS}=dd^cu_{FS}$. Without loss of generality we can assume that $$u_P+\phi=T(u_{FS}+\phi')$$ on $\Omega_t$. It now follows from the proof of Lemma \ref{lem:locality} that for all $s\in(0,t/T)$ we have that 
\begin{equation} \label{eq:last1}
u_P+\psi_{Ts}=T(u_{FS}+\psi'_s)    
\end{equation} 
on $\Omega_t$.

From the definition (\ref{def:ham2}) of $\tilde{\Phi}$ as the Legendre transform of $\psi_t$ it is easy to see that\begin{equation*} \label{thm58i}\tilde{\Phi}(z,\tau)=\psi_{Ts}(z)-Ts\ln |\tau|^2\end{equation*} iff $$Ts\in [H(z,\tau), H_-(z,\tau)].$$
Similarly letting $H'(z,\tau) = \frac{\partial}{\partial r^+}\tilde{\Phi}'(z,e^{-r/2})$ and $H_-'(z,\tau) = \frac{\partial}{\partial r^-}\tilde{\Phi}'(z,e^{-r/2})$ where $r:=-\ln|\tau|^2$ we have
\begin{equation*}\tilde{\Phi}'(z,\tau)=\psi'_{s}(z)+(1-s)\ln |\tau|^2\label{thm58ii}\end{equation*}
iff $$s-1\in [H'(z,\tau),H'_-(z,\tau)].$$  
Combined with (\ref{eq:last1}) we get that 
\begin{equation} \label{eq:thmC}
u_P(z)+\tilde{\Phi}(z,\tau)=Tu_{FS}(z)+T\tilde{\Phi}'(z,\tau)-T\ln|\tau|^2
\end{equation}
on the set 
$$U_t:=(\Omega_t\times \mathbb{D}^{\times})\cap \{(z,\tau): H(z,\tau)<t \text{ and } H'(z,\tau)<t/T-1\}.$$ 
We saw that $\tilde{\Phi}'$ was a regular solution to the HCMA on $\mathbb{P}^1\times \mathbb{D}^{\times}$ so it follows from (\ref{eq:thmC}) that $\tilde{\Phi}$ is a regular solution to the HCMA on $U_t$. 

Now $\Omega_t\times \mathbb{D}^{\times}$ exhausts $\mathbb{D}\times\mathbb{D}^{\times}$ and clearly so do the sets $\{(z,\tau): H(z,\tau)<t\}$. From Theorem \ref{RWmain} we see that $\{(z,\tau): H'(z,\tau)<t/T-1\}$ is equal to the union of the graphs over $\mathbb{D}^{\times}$ of Riemann mappings of $\Omega_s$, $s<T$, mapping zero to zero, and so $\{(z,\tau): H'(z,\tau)<t/T-1\}$ also exhausts $\mathbb{D}\times\mathbb{D}^{\times}$. It follows that $U_t$ exhausts $\mathbb{D}\times\mathbb{D}^{\times}$. Since $\tilde{\Phi}$ is a regular solution to the HCMA on $U_t$ as $t$ was chosen arbitrary we thus get that $\tilde{\Phi}$ is a regular solution to the HCMA on the whole $\mathbb{D}\times \mathbb{D}^{\times}$. 

\end{proof}

\section{Explicit Singularities}

We now give a proof of Theorem \ref{thm:notc2}  and show that a potential whose Hele-Shaw flow that develops a tangency along a set $S$ gives a singularity of the associated weak solution.  

\begin{example}\label{example}
The reader may find the following simple example instructive.  Suppose a Hele-Shaw flow  $\{\Omega_t\}$ develops tangency at a single point  $S= \{z_0\}$.   For simplicity assume there are smooth coordinates $(x,y)$ centered at $z_0$ so that near $z_0$ we have
$$ \Omega_t = \{ y \ge x^2 + (t_0-t) \} \cup \{ y \le -x^2 - (t_0-t)\}$$
giving
$$\partial \Omega_t = \{ y = x^2 + (t_0-t) \} \cup \{ y = -x^2 - (t_0-t)\}.$$
Thus $\partial \Omega_t$ consists of two disjoint parabola for $t<t_0$ that meet at the point $x=y=0$ as $t\to t_0$ from below.  Now let $$h(x,y):= H((x,y),1)$$ where $H$ is as defined in \eqref{def:ham} so by Proposition \ref{prop:ham}
$$h(x,y) = \sup\{t : (x,y)\notin \Omega_t\} -1.$$
Notice that if $y>0$ and $t\le t_0$ then $(0,y)\in \Omega_t$ if and only if $y\ge (t_0-t)$.  Thus
$$ h(0,y) = t_0-y-1 \text{ for } y>0.$$
Similar considerations for $y<0$ then give
$$ h(0,y) = \left \{ \begin{array}{cc} t_0-y-1& y>0 \\ t_0+y-1 & y<0\end{array}\right.$$
and so $\frac{\partial h}{\partial y}$ does not exist at the origin.  Hence $\tilde{\Phi}$ is not $C^2$ at the point $(z_0,1)$, and by Proposition \ref{prop:equiv} the same must be true for $\Phi$.
\end{example}

\begin{customthm}{B}\label{thm:notc2:repeat}
Let $S$ be a finite union of points and curve segments in $\mathbb P^1\setminus\{0\}$.  Let $\phi\in C^{\infty}(\mathbb P^1)$ be a K\"ahler potential and suppose the Hele-Shaw for $\omega + dd^c\phi$ develops tangency along $S$.   Then the weak solution $\Phi$ from \eqref{eq:weakD} to the Dirichlet problem for the HCMA on $\mathbb P^1\times \overline{\mathbb D}$ with boundary data $(z,\tau)\mapsto \phi(\tau z)$ is not twice differentiable at the points $(\tau^{-1}z,\tau)$, $z\in S, |\tau|=1$.
\end{customthm}
\begin{proof}[Proof of Theorem \ref{thm:notc2:preliminary} and Theorem \ref{thm:notc2}]
Suppose first that $\phi$ is as produced by Proposition \ref{prop:existencetangency}.   That is, we have picked points $z_i$ in each component of $\mathbb P^1\setminus \overline{\Omega_T}$ and $\pi\colon \Sigma\to \mathbb P^1\setminus \{z_i\}$ is the universal cover, and $\Omega_{t\in (T-\epsilon,T]}$ is the pushforward of a strong Hele-Shaw flow on $\Sigma$.   This implies that the normal velocity of the boundary of $\Omega_t$ as $t$ tends to $T$ from below is nowhere vanishing.  

Now let $z\in S$, so $\Omega_T$ has boundary tangent to itself at $z$, and so $\Omega_T$ splits locally into two pieces, call them $P_1$ and $P_2$.   Working on $P_1$, the combination of  Proposition \ref{prop:ham} (which says that $\partial \Omega_t$ are the level sets of $H(\cdot,1)-1$) and the fact that $\Omega_t$ varies smoothly imply the partial derivative of $H(\cdot,1)$ in the normal direction to $\Omega_t$ is strictly negative at $z$ (compare Example \ref{example}).  The analagous statement is true for $P_2$, which proves that $H$ is not differentiable  at $(z,1)$.    Thus $\tilde{\Phi}$ is not twice differentiable at the point $(z,1)$, and by Proposition \ref{prop:equiv} the same is true for $\Phi$.  Then  by $S^1$-invariance we see that $\Phi$ cannot be twice differentiable at any point of the form $(\tau^{-1}z, \tau)$ for $z\in S,|\tau|=1$. 

Now if $\phi$ is any K\"ahler potential whose Hele-Shaw $\{\Omega_t\}$ develops tangency along $S$ then it is not hard to see from the proof of Proposition \ref{prop:existencetangency} that $\{\Omega_t\}$ is the pushforward of some Hele-Shaw flow on $\Sigma$ call it $\{\Omega_t'\}$.  The hypothesis on $\Omega_T$ ensure that $\Omega_T'$ is smoothly bounded, and hence by the argument in Remark \ref{rmk:strongconverse} we conclude that the normal velocity is non-vanishing as $t$ tends to $T$ from below (the reader who prefers not to invoke this argument may prefer to make this non-vanishing as part of the hypothesis of what it means to develop a tangency along $S$).  The proof of the Theorem then follows as before.

Finally Theorem \ref{thm:notc2:preliminary} follows from Theorem \ref{thm:notc2} and Proposition \ref{prop:existencetangency}.
\end{proof}

\section{An extension}\label{sec:extension}

So far we have been working under the hypothesis that our Hele-Shaw flow $\{\Omega_t\}_{t\in (0,A)}$ is standard as $t$ tends to $0$ (and also as $t$ tends to 1 when $X=\mathbb P^1$).  We did this to ensure regularity of the associated potential near the point $0$ (resp.\ $\infty$) which we achieved by direct computation.   In this section we explain how this hypothesis can be relaxed.  For simplicity we work only with $(X,\omega) = (\mathbb P^1,\omega_{FS})$ but a similar story holds for the disc.

\begin{definition}
Let  $\Diff_0(\mathbb P^1)$ be the group of $C^{\infty}$-diffeomorphisms of $\mathbb P^1$ such that $\alpha(0)=0$. 
\end{definition}

Given $\alpha\in \Diff_0(\mathbb P^1)$ we define
\begin{equation}
\Omega_t = \alpha( B(t)) \text{ for } t\in (0,1)\label{eq:alphaflow}
\end{equation} 
where, we recall, $B(t)$ denotes the geodesic ball centred at 0 with area $t$ with respect to $\omega_{FS}$.    Clearly  $\{\Omega_t\}_{t\in (0,1)}$ is a strictly increasing, smoothly varying family of smoothly bounded simply connected domains in $\mathbb P^1$ that tends to zero as $t$ tends to 0.  We claim that for $\{\Omega_t\}_{t\in (0,1)}$ constructed in this way the conclusion of Theorem \ref{thm:main} and Theorem \ref{thm:existencepotential} still hold; that is, there exists a K\"ahler potential $\phi$ such that $\{\Omega_t\}_{t\in (0,1)}$ is the Hele-Shaw flow for $\omega_\phi$, and that weak geodesic obtained as the Legendre transform of the Hele-Shaw envelopes $\{\psi_t\}$ is regular.

We sketch why this is the case.  Observe that the only place in which we used that $\{\Omega_t\}_{t\in (0,1)}$ is standard as $t$ tends to $0$ and $1$ in the proof of Theorem \ref{thm:main} was to ensure that $\omega_{\phi}$ was a smooth strictly positive form at $0$ and at $\infty$.  So assume instead that \eqref{eq:alphaflow} holds.   Let $\alpha_0=\operatorname{id}_{\mathbb P^1}\in \Diff_0(\mathbb P^1)$,  whose associated flow is $\{ B(t)\}_{t\in (0,1)}$ which is the Hele-Shaw flow associated to $\omega_{FS}$ and, as we saw in \eqref{eq:kappanear0}, is the classical Hele-Shaw flow on $\mathbb C$ with permeability $\kappa_0(z):=\pi (1+|z|^2)^2$ .   Without loss of generality say $\alpha(z) = z + O(|z|^2)$ near $z=0$.   Then $\alpha$ is $C^{\infty}$ close to $\alpha_0$ in a neighbourhood of $z=0$ which implies that $\Omega_t$ is $C^{\infty}$-close to $B(t)$ for $t$ sufficiently small.  In turn this implies the permeability $\kappa$ as defined in \eqref{eq:HSclassical} is $C^{\infty}$ close to $\kappa_0$ in a punctured neighbourhood of $0$, which is enough to imply it extends across $0$ to a smooth strictly positive function.   The argument near $\infty$ is similar: again without loss of generality say $\alpha(\infty) = \infty$ locally given by $\alpha(1/z) = 1/z + O(1/|z|^2)$ near $z=\infty$.  Given a small neighbourhood $U$ of $\infty$ we can construct an $\alpha_1$ that is equal to $\alpha$ on $\mathbb P^1\setminus U$ and is equal to $\alpha_0$ near $\infty$.  Thus $\alpha$ is $C^{\infty}$ close to $\alpha_1$ and the same argument then applies to deduce that $\phi$ extends smoothly across $\infty$ and $\omega_{\phi}$ is K\"ahler.  Hence Theorem \ref{thm:main} still holds.

The argument for Theorem \ref{thm:existencepotential} is similar, as near $\{0\}\times\overline{\mathbb D}$ the foliation by harmonic discs constructed in the proof of Theorem \ref{thm:main:repeat} for $\alpha$ is (in the obvious sense) $C^{\infty}$-close to that provided by $\alpha_0$, and this is enough to prove that $\tilde{\Phi}$ is smooth over $\{0\}\times\overline{\mathbb D}$.  Arguing similarly with $\alpha_1$ near $\{\infty\}\times\overline{\mathbb D}$ we conclude that Theorem \ref{thm:existencepotential} still holds as well.

Accepting this argument, we see that to any $\alpha\in \Diff_0(\mathbb P^1)$ we have an associated smooth geodesic ray in the space of K\"ahler metrics on $\mathbb P^1$ that starts at $\omega_\phi$ and has limit $\omega + dd^c \ln |z|^2$   at infinity (i.e.\ as $\tau$ tends to zero).  Of course different $\alpha$ can give rise to the same flow, but the ambiguity is precisely coming from the subgroup of ``angular diffeomorphisms'' given by
$$\Gamma:= \{ \alpha\in \Diff_0(\mathbb P^1) : \alpha(B(t)) = B(t) \text{ for all } t\}.$$

Moreover this process can be reversed, since any smooth geodesic joining $\omega_\phi$ to $\omega + dd^c \ln |z|^2$ comes from a regular solution to the complex Monge-Amp\`ere Equation, and thus gives rise to a foliation by harmonic discs.    By the harder direction of \cite[Theorem 3.1]{RW} we know that such discs can  only be those described in  the proof of Theorem \ref{thm:main:repeat}.    Finally it is clear from the proof of Theorem \ref{thm:existencepotential} that different Hele-Shaw flows give rise to different potentials $\phi$ and vice versa.      Thus  in all we have the following explicit description of all smooth geodesics rays in the space of K\"ahler metrics on $\mathbb P^1$ that have limit $\omega_{FS} + dd^c \log |z|^2$ as time tends to infinity:

\begin{theorem}
The duality that associates a weak geodesic ray to the Hele-Shaw flow gives a bijection between $ \Diff_0(\mathbb P^1)/\Gamma$ and
$$ \{ \phi\in \mathbb C^{\infty}(X) :  \exists \text{ a smooth geodesic ray starting at }\omega_\phi  \text{ with limit  } \omega_{FS} + dd^c\ln |z|^2 \}.$$
\end{theorem}

\appendix
\section{Smoothness of Green's Functions}\label{appendix}

We first collect some regularity results for elliptic operators, all of which is essentially standard.  Suppose $I\subset \mathbb R$ is an open interval and $\{L_t\}_{t\in I}$ is a smoothly varying family of strictly elliptic operators on the unit disc $\mathbb D$ with uniform ellipticity constant.  That is, we suppose
\begin{equation}
L_t u = a^{ij}(x,t) D_{ij} u + b^i(x,t) D_i u + c(x,t) u\text{ for }t\in I\label{eq:ellipticfamily}
\end{equation}
where $a^{ij}, b^i,c \in C^{\infty}(\overline{\mathbb D}\times I)$ and $u$ is a function defined on $\mathbb D$, such that there is a $\lambda>0$ so that $a^{ij}(x,t)\zeta_i\zeta_j \ge \lambda |\zeta|^2$ for all $(x,t)\in \mathbb D\times I$ and $\zeta\in \mathbb R^N$.  We assume also $c(x,t)\le 0$ for $(x,t)\in \mathbb D\times I$.

Suppose now $\varphi\in C^{\infty}(\partial \mathbb  D\times I)$, and we write $\varphi_t(\cdot) = \varphi(\cdot,t)$.   Then for each $t\in I$ standard elliptic theory says \cite[Corollary 6.9, Theorem 6.19]{Gilbarg} there exists a unique $u_t\in C^{\infty}(\overline{\mathbb D})$ that solves
$$ L_t u_t = 0 \text{ and } u_t|_{\partial \mathbb D} = \varphi_t.$$
We claim that $u_t$ is also smooth in the $t$-variable.   It is sufficient to show it is smooth at any given fixed point $t_0\in I$, and replacing $t$ with $t-t_0$ we may assume $t_0=0$.   Then expanding $a^{ij},b^i,c$ in $t$ around $0$ we can write
$$ L_tu = L_0u + t M_1u + \cdots + t^N M_Nu + O(t^{N+1})u$$
for some operators $M_i$ that are independent of $t$.  Here and henceforth we work in the $C^{\infty}$-topology so the $O(t^{N+1})$ error terms means that for all $k\in \mathbb N$ there exists a $C_k$ such that this term is bounded by $C_k |t|^{N+1}$ in the $C^{k}(\overline{\mathbb D})$-norm.  We wish to find an expansion for $u_t$ in $t$,  say
\begin{equation}\label{eq:expansionut}
 u_t = u_0 + t v_1 + \cdots + t^N v_N + O(t^{N+1})
 \end{equation}
where $v_i \in C^{\infty}(\overline{\mathbb D})$.  To do so expand $\varphi = \varphi_0 + t \sigma_1 + \cdots + t^N \sigma_N + O(t^{N+1})$ where $\sigma_i\in C^{\infty}(\partial \mathbb D)$.  Then comparing coefficients of $t$ forces the $v_i$ to satisfy
\begin{align*}
 L_0 v_1 + M_1 u_0=0 &\text{ and } v_1 |_{\partial \mathbb  D} = \sigma_1\\
 L_0 v_2 + M_1 v_1 + M_2 u_0=0 &\text{ and } v_2 |_{\partial \mathbb D} = \sigma_2
 \end{align*}
and so forth.  So starting with $u_0$ we may inductively define $v_i$, and as $L_0$ is elliptic, the same elliptic regularity guarantees $v_i\in C^{\infty}(\overline{\mathbb D})$.    To see that \eqref{eq:expansionut} does actually hold, observe that by construction the difference $w_t := u_t - u_0 -tv_1  -\cdots v_Nt^N$ satisfies
$$ L_t w_t = O(t^{N+1}) \text{ and } w_t |_{\partial \mathbb D} = O(t^{N+1}).$$
Then, by elliptic theory again \cite[Corollary 8.7, Theorem 8.13]{Gilbarg} this implies $w_t = O(t^{N+1})$ in the $C^{\infty}(\overline{\mathbb D})$ topology (here we are using that the elliptic constant for $L_t$ is uniform over $t\in I$ to apply \cite[Corollary 8.7]{Gilbarg} uniformly over $I$), which gives \eqref{eq:expansionut}.  As this holds for all $N$,  the map $t\mapsto u_t$ is smooth in $t$, which implies $u\in C^{\infty}(\overline{\mathbb D}\times I)$ as claimed.

\begin{theorem}\label{thm:appmain}
Let $I\subset \mathbb R$ be an open interval, and assume that $\{\Omega_t\}_{t\in I}$ is a smoothly varying family of smoothly bounded simply connected domains in $\mathbb C$.   Let $\zeta$ be a function that is smooth on a neighbourhood of 
$$ \overline{\bigcup_I \partial \Omega_t},$$
and for each $t\in I$ let $v_t$ be the solution to the Dirichlet problem
\begin{equation}\Delta v_t=0 \text{ and } v_t|_{\partial \Omega_t} = \zeta|_{\partial \Omega_t}\label{eq:appdirichlet}
\end{equation}
Then $v_t$ varies smoothly with $t$.
\end{theorem}
\begin{proof}
We have that $\Omega_t = \alpha_t(\mathbb D)$ where $\alpha\colon \overline{\mathbb D}\times I\to \mathbb C$ is smooth and each $\alpha_t:\overline{\mathbb D} \to \overline{\Omega}_t\subset \mathbb C$ is a diffeomorphism.  Set
$$L_t(\tilde{u}) := (\Delta (\tilde{u}\circ \alpha_t^{-1})) \circ \alpha_t$$
where $\tilde{u}:\overline{\mathbb D}\to \mathbb R$ and $\Delta$ is the standard Laplacian on $\mathbb C$.  Thus if 
$$ \tilde{u} = \tilde{v}\circ \alpha_t$$
then
$$ L_t(\tilde{u})|_x = \Delta(\tilde{v})|_{\alpha_t(x)}.$$

Clearly $\{L_t\}_{t\in I}$ is a smoothly varying family of elliptic operators with uniform ellipticity constant, as in \eqref{eq:ellipticfamily}. Define $\varphi\in C^{\infty}(\partial \mathbb D\times I)$ by
$$\varphi(z,t):= \zeta(\alpha_t(z)).$$
Then by the above discussion we know there exists a $u\in C^{\infty}(\overline{\mathbb D}\times I)$ such that
$$ L_t u_t =0 \text{ and } u_t|_{\partial \mathbb D} = \varphi(\cdot,t).$$
Hence
$$v(z,t) := u(\alpha_t^{-1}(z),t)$$
satisfies \eqref{eq:appdirichlet}, and varies smoothly in $t$.
\end{proof}

\begin{corollary}\label{cor:ptsmoothint}
Assume in addition that each $\Omega_t$ contains the origin.  Then the pressure $p_t$ which satisfies
$$\Delta p_t = -\delta_0 \text{ and } p_t|_{\partial \Omega_t}=0$$
varies smoothly in $t$.
\end{corollary}  
\begin{proof}
Apply Theorem \ref{thm:appmain} to $\zeta(z) :=  \log |z|^2$ and let $p_t = v_t - \log |z|^2$.
\end{proof}

\begin{remark}
Hence the quantity $\nabla p_t$ on $\partial \Omega_t$ is a smooth vector field on $\cup_{t\in I} \partial \Omega_t$ which is precisely what we used in the proof of Theorem \ref{thm:existencepotential}.
\end{remark}

\begin{corollary}\label{cor:riemannmapssmoothint}
Continue to assume that each $\Omega_t$ contains the origin.  Then there is a family of Riemann maps $f_t: \mathbb D\to \Omega_t$ that vary smoothly with $t$.
\end{corollary}
\begin{proof}
This is just the standard way of constructing Riemann maps from solutions to the Dirichlet problem.  In fact if $\Delta v_t=0$ on $\Omega_t$ and $v_t = \log |z|$ on $\partial \Omega_t$ we let
$$ g_t(z) = ze^{v_t + iw_t}$$
where $w_t$ is a harmonic conjugate to $v_t$ (i.e.\ chosen so $v_t+iw_t$ is holomorphic).  Then $g_t:\Omega_t\to \mathbb D$ is a holomorphic map taking $\partial \Omega_t$ to $\partial \mathbb D$.  One shows that moreover $g_t:\Omega_t\to \mathbb D$ is a biholomorphism, and the Riemann map $f_t := g_t^{-1}$ varies smoothly with $t$ as $v_t$ does.

\end{proof}

\medskip
\small{
\noindent {\sc Julius Ross,  DPMMS , University of Cambridge, UK. \\
j.ross@dpmms.cam.ac.uk}\medskip

\noindent{\sc David Witt Nystr\"om, DPMMS,  University of Cambridge, UK. \newline d.wittnystrom@dpmms.cam.ac.uk, danspolitik@gmail.com}

}


\begin{thebibliography}{99}

\bibitem{ArezzoTian} C.\ Arezzo and G.\ Tian \emph{Infinite geodesic rays in the space of K\"ahler potentials.} Ann. Sc. Norm. Super. Pisa Cl. Sci. (5) 2 (2003), no. 4, 617--630.

\bibitem{Bedford} E.\ Bedford and J.-P. Demailly \emph{Two counterexamples concerning the pluri-complex Green function in ${\bf C}^n$.} Indiana Univ. Math. J. 37 (1988), no. 4, 865--867. 

\bibitem{BedfordTaylor} E. Bedford and B. A. Taylor \emph{The Dirichlet problem for a complex Monge-Amp\`ere equation.} Invent. Math. 37 (1976), no. 1, 1--44.

\bibitem{Blocki} Z.\ B\l ocki \emph{The $C^{1,1}$ regularity of the pluricomplex Green function.} Michigan Math. J. 47 (2000), 211--215.

\bibitem{Blocki2} Z.\ B\l ocki  \emph{On geodesics in the space of K\"ahler metrics.} Advances in geometric analysis, 3--19, Adv. Lect. Math. (ALM), 21, Int. Press, Somerville, MA, 2012. 

\bibitem{Chen} X.X.\ Chen  \emph{The space of K\"ahler metrics} J. Differential Geom. 56 (2000), no. 2, 189--234. 

\bibitem{Darvas} T.\ Darvas \emph{Morse theory and geodesics in the space of K\"ahler metrics.} Proc.\ Amer.\ Math.\ Soc.\ 142 (2014), no. 8, 2775--2782. 

\bibitem{Don3} S.\ Donaldson \emph{Symmetric spaces, K\"ahler geometry and Hamiltonian dynamics.} Northern California Symplectic Geometry Seminar, 13--33,  Amer. Math. Soc. Transl. Ser. 2, 196, Amer. Math. Soc., Providence, RI, 1999.

\bibitem{Donaldson}  S.\ Donaldson \emph{Holomorphic discs and the complex Monge-Amp\`ere equation}.  J. Symplectic Geom. 1 (2002), no. 2, 171--196. 

\bibitem{Gustafsson} B.\ Gustafsson and A.\  Vasil'ev \emph{Conformal and potential analysis in Hele-Shaw cells.} Advances in Mathematical Fluid Mechanics. Birkh\"auser Verlag, Basel, 2006. x+231.

\bibitem{Gilbarg} D.\ Gilbarg and N.\ Trudiger \emph{Elliptic partial differential equations of second order.} Reprint of the 1998 edition. Classics in Mathematics. Springer-Verlag, Berlin, 2001. xiv+517 pp. ISBN: 3-540-41160-7. 

\bibitem{Hedenmalm} H.\ Hedenmalm and S.\ Shimorin \emph{Hele-Shaw flow on hyperbolic surfaces.} J. Math. Pures Appl. (9) 81 (2002), no. 3, 187--222.

\bibitem{HedenmalmOlofsson} H.\ Hedenmalm and A.\ Olofsson \emph{Hele-Shaw flow on weakly hyperbolic surfaces.} Indiana Univ. Math. J. 54 (2005), no. 4, 1161--1180.  

\bibitem{Krantz} S.\ Krantz \emph{Function theory of several complex variables.} Reprint of the 1992 edition. AMS Chelsea Publishing, Providence, RI, 2001. 

\bibitem{Lempert} L.\ Lempert \emph{La m\'etrique de Kobayashi et la repr\'esentation des domaines sur la boule.} Bull. Soc. Math. France 109 (1981), no. 4, 427--474.


\bibitem{LempertVivas} L.\ Lempert and L.\ Vivas \emph{Geodesics in the space of K\"ahler metrics.}
Duke Math. J. 162 (2013), no. 7, 1369--1381.

\bibitem{LempertDarvas}L. Lempert and T. Darvas \emph{Weak geodesics in the space of K\"ahler metrics} Math. Res. Lett. 19 (2012), no. 5, 1127--1135. 

\bibitem{Mabuchi} T.\ Mabuchi
\emph{Some symplectic geometry on compact K\"ahler manifolds. I.} Osaka J. Math. 24 (1987), no. 2, 227--252.

\bibitem{Richardson} S.\ Richardson \emph{Hele-Shaw flows with a free boundary produced by the injection of fluid into a narrow channel.} J. Fluid Mech., 56 (1972), no. 4, 609-618.

\bibitem{RossNystromTubular} J.\ Ross and D.\ Witt Nystr\"om \emph{Homogeneous Monge-Amp\`ere Equations and Canonical Tubular Neighbourhoods in K\"ahler Geometry.}  IMRN, Volume 2017, Issue 23 (2017), Pages 7069--7108.


\bibitem{RWDisc} J.\ Ross and D.\ Witt Nystr\"om \emph{The Hele-Shaw flow and Moduli of Holomorphic Discs.} Compos.\ Math.\ 151 (2015), no.\ 12, 2301--2328. 

\bibitem{RW} J.\ Ross and D.\ Witt Nystr\"om \emph{Harmonic Discs of Solutions to the Complex Homogeneous Monge-Amp\`ere Equation.}  Publ.\ Math.\ Inst.\ Hautes \`{E}tudes Sci.\ 122 (2015), 315--335. 


\bibitem{RubinsteinZelditchI} Y.\ Rubinstein and S.\ Zelditch \emph{The Cauchy problem for the homogeneous Monge-Amp\`ere equation, I. Toeplitz quantization.} J. Differential Geom. 90 (2012), no. 2, 303--327. 

\bibitem{RubinsteinZelditchII} Y.\ Rubinstein and S.\ Zelditch  \emph{The Cauchy problem for the homogeneous Monge-Amp\`ere equation, II. Legendre transform.}  Adv. Math. 228 (2011), no. 6, 2989--3025. 

\bibitem{RubinsteinZelditchIII} Y.\ Rubinstein and S.\ Zelditch \emph{The Cauchy problem for the homogeneous Monge-Amp\`ere equation, III. Lifespan.} J. Reine Angew. Math. 724 (2017), 105--143.

\bibitem{Semmes} S.\ Semmes \emph{Complex Monge-Amp\`ere and symplectic manifolds} Amer. J. Math. 114 (1992), no. 3, 495--550.

\end{thebibliography}
\end{document}